\documentclass[a4paper,11pt]{article}
\usepackage{showkeys}
\usepackage[english]{babel}
\usepackage{graphicx}

\usepackage{amsxtra}
\usepackage{amssymb}
\usepackage{amsmath}
\usepackage{amsthm}
\usepackage{color}
\usepackage{hyperref}

\newtheorem{prop}{Proposition}[section]
\newtheorem{define}[prop]{Definition}

\newtheorem{lemma}[prop]{Lemma}
\newtheorem{theo}[prop]{Theorem}

\numberwithin{equation}{section}

\theoremstyle{remark}

\newcommand{\cxi}{\langle \xi\rangle}

\newcommand{\R}{\mathbb{R}}
\newcommand{\N}{\mathbb{N}}
\newcommand{\Z}{\mathbb{Z}}

\title{On the sharp regularity 
of solutions to hyperbolic boundary 
value problems}
\author{Corentin Audiard
\footnote{Sorbonne Universit\'es, UPMC Univ Paris 06, UMR 7598, Laboratoire Jacques-Louis Lions, 
F-75005, Paris, France }
\footnote{CNRS, UMR 7598, Laboratoire Jacques-Louis Lions, F-75005, Paris, France}}
\begin{document}
\maketitle
\begin{abstract}We prove some sharp regularity results for solutions 
of classical first order hyperbolic initial boundary value problems. Our two 
main improvements on the existing litterature are weaker regularity assumptions for the 
boundary data and regularity in fractional Sobolev spaces. This 
last point is specially interesting when the regularity index 
belongs to $1/2+\N$, as it involves nonlocal compatibility 
conditions.
\end{abstract}
\section{Introduction}
\paragraph{Everything in a toy model}
Consider the simplest hyperbolic initial boundary value problem (IBVP)
$$
\left\{
\begin{array}{ll}
\partial_tu+\partial_xu(x,t)=0,\ (x,t)\in (\R^+)^2\\ 
u(x,0)=u_0(x),\\ 
u(0,t)=g(t)
\end{array}\right.
$$
When $(u_0,g)\in L^2(\R^+)^2$, the  solution 
is piecewise defined: $u(x,t)=u_0(x-t)$ for $x-t\geq 0$, 
$g(t-x)$ for $x-t< 0$, it belongs to $C_tL^2$. \\
It is well known that the smoothness of $(u_0,g)$ is not enough to 
ensure the smoothness of $u$, compatibility conditions are required: for $k\in \N$, $u\in \cap_{j=0}^kC_t^jH^{k-j}$ if and only if
$$
(u_0,g)\in (H^k)^2\text{ and }\forall\,j\leq k-1,\ u_0^{(j)}(0)=(-1)^jg^{(j)}(0).
$$
These compatibility relations are trivial here due to the solution 
formula, but are more generally derived considering $u$ 
(and its derivatives) at the corner $x=t=0$, and
writing $\partial^\alpha u|_{x=0}|_{t=0}=\partial^\alpha u|_{t=0}|_{x=0}$. A basic rule of thumb 
being that any compatibility condition that makes sense should be 
true.\\
For fractional regularity, not much changes 
except in the notoriously pathologic case $s\equiv 1/2[\Z]$. Indeed even if there is no trace in $H^{1/2}(\R^+)$, 
the gluing of two functions 
in $H^{1/2}(\R^+)$ is not $H^{1/2}(\R)$. 
The simplest way to see this is 
to consider the map 
$f\in L^2(\R)\to f(\cdot)-f(-\cdot)\in L^2(\R^+)$. It is continuous 
$L^2(\R)\to L^2(\R^+)$ and $H^1\to H^1_0(\R^+)$ hence 
$H^{1/2}(\R)\to [L^2,H^1_0]_{1/2}$ by 
interpolation. The interpolated space is the famous
Lions-Magenes space $H^{1/2}_{00}(\R^+)$, and it is different 
(algebraically and topologically) from $H^{1/2}(\R^+)$: by interpolation of 
Hardy's inequality, any function 
$f\in H^{1/2}_{00}(\R^+)$ must satisfy
$$
\int_{\R^+}\frac{f^2(x)}{x}dx<\infty,
$$
this is obviously not the case for 
functions merely in $H^{1/2}(\R^+)$.\\
For the regularity of solutions of the BVP, 
this adds a ``global'' compatibility condition 
$$
u\in C_tH^{1/2}\Leftrightarrow (u_0,g)\in 
H^{1/2}(\R^+)\text{ and }
\int_{\R^+}\frac{|g(x)-u_0(x)|^2}{x}dx<\infty.
$$
Our aim here is to prove an analogous 
result for general hyperbolic boundary 
value problems.
\paragraph{Settings and results}
Let $\Omega$ be a smooth open set of $\R^d$,
we consider first order boundary value 
problems of the form 
\begin{equation}\label{IBVP}
\left\{
\begin{array}{ll}
Lu:=(\partial_t-\sum_{j=1}^d A_j\partial_j)u=0,\ 
(x,t)\in \Omega\times \R_t^+,\\
Bu|_{\partial\Omega}=g,\ (x,t)\in \partial\Omega\times \R_t^+,\\
u|_{t=0}=u_0,\ x\in \Omega.
\end{array}
\right.
\end{equation}
The index $t$ in $\R_t^+$ has no meaning 
except to emphasize the time variable.
The $A_j's$ are $q\times q$ matrices 
depending smoothly on $(x,t)$, $B$ is a
smooth $b\times q$ matrix, $b$ is the number 
of boundary conditions.\\
For data $(u_0,g,f)\in L^2(\Omega)\times
L^2(\partial\Omega\times \R_t^+)\times 
L^2(\R^+_t\times \Omega)$, the 
well-posedness of such hyperbolic
BVP has been obtained in a large variety 
of settings, that we will only shortly 
mention. After the pioneering results 
of Friedrichs \cite{Friedrichs} for symmetric 
dissipative systems, Kreiss \cite{Kreiss} 
proved the
well-posedness of the BVP with zero initial 
data in the strictly hyperbolic 
case ($\sum A_j\xi_j$ has only real 
eigenvalues of algebraic multiplicity one)
under the now standard Kreiss-Lopatinskii 
condition on $B$. In Kreiss's framework, the case of $L^2$ initial data was
then tackled by Rauch \cite{Rauch}.  
Well-posedness of constantly hyperbolic BVP 
was later obtained by M\'etivier 
\cite{Met} (zero initial data), the author 
then proved well-posedness with $L^2$ initial 
data \cite{Audiard1}.
A further generalization was obtained 
by M\'etivier \cite{Metsemigroupe}
for a new class 
of hyperbolic operators, 
larger than the constantly hyperbolic ones.
He also gave a new proof, both more 
general and simpler, of well-posedness with $L^2$ initial data.\\
For more references and results, in particular for 
characteristic BVP (that we do not consider here) the reader may refer 
to the book \cite{Benzoni3}.\\
Let $\mathrm{n}$ be a normal on $\partial\Omega$, 
the problem \eqref{IBVP} is said 
to be noncharacteristic when 
$\sum A_jn_j$ is invertible on 
$\partial\Omega$. 
For non characteristic boundary value 
problems, the main reference on the 
smoothness of solutions is the classical 
paper of Rauch and Massey \cite{RauchMassey}, where, under 
no specific assumption (except of course 
well-posedness), the authors prove that the 
solution of \eqref{IBVP} belongs to 
$\cap_{j=0}^k C^j_t(\R^+_t,H^{k-j}(\Omega))$
when $(u_0,g,f)\in H^k(\Omega)\times H^{k+1/2}(\partial\Omega\times\R_t^+)\times H^k(\Omega\times\R_t^+)$ and
satisfy natural compatibility conditions
that we describe now. For conciseness, 
when there is no ambiguity we will 
usually denote $H^k$ instead of 
$H^k(X)$, $X=\Omega,\partial\Omega\times \R_t^+,\Omega\times \R_t^+$.\\
We denote $\mathcal{A}=\sum A_j\partial_j$
and define inductively $v_j$ the formal 
value of $(\partial_t^ju)|_{t=0}$ by 
\begin{equation}\label{taylor}
v_0=u_0,\ v_{j+1}=(\partial_t^j\partial_tu)|_{t=0}
=\partial_t^j(\mathcal{A}u+f)|_{t=0}
=\sum_{l=0}^j\binom{j}{l}
(\partial_t^l\mathcal{A}|_{t=0})v_{j-l}+\partial_t^jf|_{t=0}.
\end{equation}
The first order compatibility condition is 
$Bv_0|_{\partial\Omega}=g|_{t=0}$ and the 
generic compatibility 
condition of order $j$ is 
\begin{equation}
\label{CCj}
\text{Compatibility at order }j:\ 
\partial_t^{j-1}g|_{t=0}=\sum_{l=0}^{j-1}
\binom{j-1}{l}(\partial_t^lB)v_{j-1-l}|_{\partial\Omega}.
\end{equation}
Note that \eqref{CCj} makes sense as soon 
as $(u_0,g,f)\in (H^s)^3,\ s>j-1/2$.
If the smoothness of the data is $j-1/2,
j\in \N^*$,
we define a special compatibility 
condition : when
$\Omega=\R^{d-1}\times \R^+$, 
denote $x=(x',y)$;
the condition is \\
Compatibility at order $j-1/2$:
\begin{equation}
\label{CCj-}
\partial_t^{j-1}g(x',t)-\left(\sum_{l=0}^{j-1}
\binom{j-1}{l}(\partial_t^lB)v_{j-1-l}(x',t)\right)
\in H^{1/2}_{00}\left(\R^{d-1}\times (\R^+)\right). 
\end{equation}
For general smooth $\Omega$, \eqref{CCj-} 
is defined similarly through local maps
and a partition of unity: near the boundary $\Omega$ is diffeomorphic to 
(a part of) $\R^{d-1}\times \R^+$ 
thanks to some map 
$\Phi$, one simply requires 
\eqref{CCj} to stand for $g(\Phi(x',0),t),
(v_l\circ\Phi(x',t))_{0\leq l\leq j-1}$.
\\
Note that due to Hardy's inequality, the 
$j$-th condition implies the condition of 
order $j-1/2$.
\begin{define}
If $s=k+\theta,\ -1/2< \theta <1/2,\ 
k\in \N^*,\ 
\theta\neq 1/2$, we say that data 
$(u_0,g,f)\in  (H^s)^3$ satisfy the 
compatibility conditions at order $s$ when \eqref{CCj}
is satisfied for $1\leq j\leq k$.\\
If $s=k-1/2$, the compatibility conditions 
are satisfied at order $s$ when 
\eqref{CCj} is true for 
$1\leq j\leq k-1$ and \eqref{CCj-} is true 
for $j=k$.
\end{define}
A strong $L^2$ solution of \eqref{IBVP} 
is a function $u\in C_tL^2$ such that 
there exists a sequence $u_n$ of smooth 
solutions of \eqref{IBVP} with data 
$(u_{0,n},g_n,f_n)$ that converge to 
$(u_{0},g,f)$ in $L^2$, and for any 
$T>0$, $\|u-u_n\|
_{C([0,T],L^2)}\to 0$.
\subparagraph{Assumptions}
We need the smoothness
of $\Omega$ and the well-posedness of 
\eqref{IBVP}:
\begin{enumerate}
\item \label{assump1} $\partial\Omega$ 
is a smooth hypersurface
with normal $\nu$, parametrized 
by local maps $(\phi_j(y'))_{1\leq j\leq J}$,
$y'\in \R^{d-1}$, and 
$\varphi_j(y',y_d):=\phi_j(y')+y_d\nu(\varphi_j(y'))$ are local 
diffeomorphisms $V_j\to U_j$, with 
$\varphi_j
 ((\R^{d-1}\times \R^{+*})\cap V_j) \subset 
 \Omega$, and $\cup_{j=1}^JU_j\supset \partial\Omega$. 
 We do 
 not assume that the $U_j$ are bounded sets,
 but $D\varphi_j,D\varphi_j^{-1}$ must be 
 uniformly bounded, and $d(\Omega\setminus 
 \cup \text{Im}(\varphi_j),\,
 \partial\Omega)>0$.
 \item \label{charac} 
 The boundary is uniformly not characteristic, in the sense that 
$\sum A_j\nu_j$ is invertible on 
$\partial\Omega$, and the inverse 
is uniformly bounded.
 \item 
For data $(u_0,g,f)\in (L^2)^3$, there 
exists a unique strong $L^2$ 
solution\footnote{This assumption is 
somewhat too strong, as it is classical 
that in this framework, weak solutions are actually strong, see 
\cite{LaxPhillips}.} to 
\eqref{IBVP} that satisfies the semi-group estimate
for $\gamma$ large enough
\begin{eqnarray}\nonumber
\|e^{-\gamma \cdot}u\|_{C([0,t],L^2(\Omega))}+\sqrt{\gamma}
 |e^{-\gamma \cdot}u|_{\partial\Omega}|_{L^2(\partial\Omega\times
 [0,t])}&\lesssim& \|u_0\|_{L^2(\Omega)}+
 |e^{-\gamma \cdot }g|_{L^2(\partial\Omega\times [0,T])}
 \\
 \label{semigroupe}
 &&+\frac{\|e^{-\gamma \cdot}f\|_{L^2
 ([0,t]\times \Omega)}}{\sqrt{\gamma}}.
 \end{eqnarray}
We use the convention that norms inside 
the domain are denoted $\|\cdot\|$ while 
norms on the boundary are denoted $|\cdot|$.
\end{enumerate}
We point out that a consequence of the semi-group estimate is 
the \emph{resolvent estimate}: for $\gamma$ large enough 
(larger than for \eqref{semigroupe})
\begin{eqnarray}
\nonumber
\gamma \|e^{-\gamma t}u\|_{L^2(\Omega\times \R_t^+)}^2
&+&|e^{-\gamma t}u|_{\partial\Omega}|_{L^2(\partial\Omega\times \R_t^+)}^2
\\
\label{resolv}
&\lesssim & \left(\|u_0\|_{L^2(\Omega)}^2+
 |e^{-\gamma t}g|_{L^2(\partial\Omega\times \R_t^+)}^2+\frac{
 \|e^{-\gamma t}f\|_{L^2}^2}{\gamma}\right).
\end{eqnarray}
This is readily obtained by squaring \eqref{semigroupe} for 
some fixed $\gamma_0$, multiplication by $e^{-2\gamma t},\ 
\gamma >\gamma_0$ and integration in $t$.
Higher regularity versions of the resolvent and the semi-group
estimates are a bit more delicate to state. We define 
weighted Sobolev spaces $H^s_\gamma$ in section
\ref{notations}, the weighted resolvent estimate is then 
\begin{eqnarray}
\label{resolvreg}
\gamma \|u\|_{H^s_\gamma}^2
+|u|_{\partial\Omega}|_{H^s_\gamma}^2
\lesssim \|u_0\|_{H^s(\Omega)}^2+
 |g|_{H^s_\gamma}^2+\frac{
 \|f\|_{H^s_\gamma}^2}{\gamma}.
\end{eqnarray}
The main point of this estimate is that it is sharp with 
respect to the parameter $\gamma$ and 
allows to absorb commutators in a priori 
estimates. Moreover, it implies 
the following (simpler to read) estimate
\begin{eqnarray}\nonumber
\|e^{-\gamma t}u\|_{H^s(\Omega\times \R^+)}^2
&+&|e^{-\gamma t}u|_{\partial\Omega}|_{H^s(\partial\Omega\times \R^+)}^2
\\
\label{resolvreguseless}
&\lesssim& \|u_0\|_{H^s(\Omega)}^2+
 |e^{-\gamma t}g|_{H^s(\partial\Omega\times \R^+)}^2+
 \|e^{-\gamma t}f\|_{H^s(\Omega\times \R^+)}^2.
\end{eqnarray}
We shall not need 
something as precise for the semi-group estimate:
let $s=k+\theta,\ 
k\in \N,\ 0<\theta<1$, then 
\begin{eqnarray}
\nonumber
\sum_{j=0}^k\|e^{-\gamma t}\partial_t^ju\|_{C(\R_t^+,H^{k-j+\theta}(\Omega))}^2
&+&|e^{-\gamma t}u|_{\partial\Omega}|_{H^s(\Omega\times \R_t^+)}^2
\\
\label{semigreg}
&\lesssim& \|u_0\|_{H^s(\Omega)}^2+
 |e^{-\gamma t}g|_{H^s(\partial\Omega\times [0,T]}^2+
 \|e^{-\gamma t}f\|_{H^s}^2.
\end{eqnarray}
Both estimates should be modified when $s=k+1/2$, 
$k\in\N$: it is necessary to add 
in the right hand 
side the $H^{1/2}_{00}$ norm of 
$\partial_t^kg-\sum_0^k
\binom{k}{l}(\partial_t^lB)v_{k-1-l}$, 
see page \pageref{cas1/2} 
for details. This 
is the (implicit) convention that we
use in theorem \ref{mainth}, we refer 
to the proof for more details.

An interesting related feature 
is that the constant 
in $\lesssim$ can not be uniform in 
$\theta$, it blows up as 
$\theta \to 1/2$ and the estimates 
are actually not true for $\theta=1/2$.

We can now state more precisely the regularity 
result of Rauch and Massey:
\begin{theo}[\cite{RauchMassey}]
If $(u_0,g,f)\in H^k(\Omega)\times H^{k+1/2}(\partial \Omega\times \R_t^+)
\times H^k(\Omega\times \R_t^+)$ 
satisfy the compatibility 
condition up to order $k$, the solution of 
\eqref{IBVP} belongs to $\cap_{j=0}^k C_t^jH^{k-j}$.
\end{theo}
The only suboptimal part of the theorem is the 
regularity assumption on $g$. This is due to the 
fact that the theorem is deduced from the 
homogeneous case $g=0$ with a lifting argument. 
It was already pointed out 
at the time by the authors that it could be 
improved (without proof), but quite unfortunately
the result that remained in the litterature is the 
suboptimal one, see for example the reference 
book \cite{Benzoni3}, and in somewhat different 
settings the lecture notes \cite{Met2} or 
the interesting discussion in 
the introduction of \cite{lanigu}.
\\
Our result is that the same property
holds with boundary data in $H^k$ instead
of $H^{k+1/2}$, moreover we allow 
$k$ to be any nonnegative real number rather 
than an integer.
\begin{theo}\label{mainth}
Let $s\in \R^+$. 
If $(u_0,g,f)\in H^s(\Omega)\times H^{s}(\partial \Omega\times \R_t^+)
\times H^s(\Omega\times \R_t^+)$ 
satisfy the compatibility 
condition up to order $s$, the solution of 
\eqref{IBVP} belongs for any $T>0$ to  $H^s(\Omega\times [0,T])$, satisfies 
estimate \eqref{resolvreg} for $\gamma$ large
enough, and if 
$s=k+\theta$, $k\in \N,\ 0\leq \theta <1$, 
it satisfies \eqref{semigreg}.
\end{theo}
The proof when $s$ is an integer is quite 
similar to the original argument of Rauch
and Massey, actually the fact that we 
handle directly nonzero boundary data 
leads to some slight simplifications
due to the fact that it allows to avoid 
a reduction to the case
where $B$ is constant. The fractional 
case is essentially an interpolation argument, however
it is not trivial due to the presence 
of the compatibility conditions. 
For example, in the model case described earlier instead of interpolating 
$[L^2\times L^2,H^1\times H^1]_\theta$
one must identify  
$[L^2\times L^2,\{(u_0,g)\in H^1(\R^+)\times H^1(\R^+):\ u_0(0)=g(0)\}]_\theta$.
\\
The litterature on such problems is not very 
rich. 
Another related problem is the interpolation
of Sobolev spaces with boundary conditions,
that are in some sense between 
$H^s$ and $H^s_0$.
This issue appeared quite long ago 
for elliptic equations 
on non smooth domains or parabolic 
problems, see e.g. the last section of \cite{Grisvard}, sections 14-17 of chapter 
$4$ in \cite{lionsmagenes2} (where most 
of the identification problems were 
left open), or the more 
recent (and much more involved) book 
\cite{Amann}, in particular VIII.2.5.
Due to the technicity of this last
reference (anisotropic Besov spaces are studied), 
degenerate cases (in our settings 
$s\in \N+1/2$) are not considered.
The Schr\"odinger equation 
on a domain and related interpolation problems were also  studied by the author 
in \cite{Audiard7}, where the natural 
spaces for the boundary data are Bourgain 
spaces.
\paragraph{Plan of the article}
Section $2$ is devoted to notations 
and a brief reminder on interpolation.
The proof of theorem \ref{mainth} 
is then organized in 
three sections : 
in section $3$ we recall a standard 
smoothness  result for the pure boundary 
value problem posed for $t\in \R$, due to 
Tartakov. For completeness, we include a 
sketch of proof that follows an argument 
of the (unfortunately depleted) book 
\cite{ChaPi}. Theorem 
\ref{mainth} in the case $s$ integer is 
proved in section $4$. An 
important point is a basic lifting lemma
which proves also useful for the general 
case. In section $5$, smoothness is first 
proved for $0\leq s\leq 1$ with an 
interpolation argument, then for any 
$s$ with a non trivial differentiation 
argument.

\paragraph{Ackowledgement} This work was partially funded
by the ANR project NABUCO, ANR-17-CE40-0025.
\section{Notations and basic results}\label{notations}
\paragraph{Basic notations}
Proofs are often reduced to the case 
$\Omega=\R^{d-1}\times \R^+$. In such 
settings, 
we denote the variable $x=(x',y)$
$x'\in \R^{d-1}$. The variables 
$x',t$ are called tangential, while $y$
is the normal variable. \\
Partial differential operators acting 
on functions of $(x,t)$ are 
written as $\partial^\alpha$, $\alpha\in 
\N^{d+1}$, by convention $\alpha_{d+1}$ 
is the order of differentation in time. 
A multi-index, or a differential operator, 
is said to be tangential when 
$\alpha_d=0$.
\\
We denote $[L_1,L_2]=L_1L_2-L_2L_1$ 
the commutator between two linear 
operators.
\paragraph{Sobolev spaces}
$\Omega$ is assumed to be a smooth open 
set as in definition page \pageref{assump1}.
The Sobolev spaces $H^s(\Omega)$, are 
defined when $s$ is an integer as 
$$
\{u\in L^2:\ \|u\|_{H^s}^2=
\sum_{|\alpha|\leq s}\int_\Omega|\partial^\alpha u|^2dx<\infty\}.
$$
When $s$ is not an integer, they are
defined by (complex) interpolation, 
$H^s=[L^2,H^k]_{s/k}$ for any integer 
$k$ larger than $s$.
This definition does not depend on $k$.
\\
The Sobolev 
spaces for functions defined on 
$\partial\Omega,\Omega\times \R_t^+$ 
etc are defined in the same standard 
way.\\
$H^s_0(\Omega)$ is the closure of $C_c^\infty(\Omega)$. We do have 
$[L^2,H^1_0]_s=H^s_0$ for $0<s<1$, except
for $s=1/2$, where $H^{1/2}_{0}=H^{1/2}$ 
and $[L^2,H^1_0]_{1/2}=H^{1/2}_{00}$ 
is different algebraically and topologically from $H^{1/2}$. It is 
a Banach space endowed with the norm
$$
\|u\|_{H^{1/2}_{00}}^2=
\|u\|_{H^{1/2}}^2+\int_{\Omega}
\frac{|u(x)|^2}{d(x)}dx,
$$
where $d$ is the distance to $\partial\Omega$ (see \cite{lionsmagenes}). 
The only important fact, regularly used
in the article, is that if $X_0,X_1$
are Banach spaces, an operator
$T:\ X_0\to L^2,\ X_1\to H^1_0$ maps
$[X_0,X_1]_{1/2}$ to $H^{1/2}_{00}$.
For example, $u\in H^s(\R^d)\to 
u(x',y)-u(x',-y)$ maps $H^{1/2}(\R^d)$
to $H^{1/2}_{00}(\R^{d-1}\times \R^+)$.
\\
The weighted Sobolev spaces $H^s_\gamma$ are defined as 
follows :
\begin{define}
 When $s$ is a nonnegative integer we 
 define $H^s_\gamma(\Omega\times \R_t^+)$ as the the set of functions in $L^2$
such that the following norm is finite
$$
\|u\|_{H^s_\gamma}=\sum_{|\alpha|\leq s}
\|e^{-\gamma t}\partial^\alpha u\|_{L^2}.
$$
When $s$ is not an integer, $H^s_\gamma$
is defined by complex interpolation : 
if $k$ is an integer larger than $s$, 
$H^s_\gamma=[L^2_\gamma,H^k_\gamma]_{s/k}$.\\
$H^s_\gamma(\partial\Omega\times \R_t^+)$ is defined 
similarly.
\end{define}
When $s$ is an integer, it is a straightforward consequence 
of Leibniz formula $\partial_t^j(e^{-\gamma t}u)=\sum \binom{j}{i} (-\gamma)^ie^{-\gamma t}\partial_t^{j-i}u$ that 
the $H^s_\gamma$
norm is equivalent to $\|e^{-\gamma t}u\|_{H^s}$, though 
with constants that depend on $\gamma$, hence the $H^s_\gamma$ 
spaces coincide algebraically and topologically with 
the set of functions such that $e^{-\gamma t}u\in H^s$.
\paragraph{Traces}
Sobolev spaces on $\partial\Omega$ are
defined with local maps.
The trace operator is an isomorphism:
$$
\left\{
\begin{array}{ll}
H^s(\Omega)\to \prod_{k< s-1/2} H^{s-1/2-k}(\partial\Omega),\\
u\to (\partial_n^ku|_{\partial\Omega})_{
k<s-1/2},
\end{array}\right.
$$
where $\partial_n$ is the normal 
derivative on $\partial\Omega$.\\
For functions defined in $H^s(\Omega\times 
\R^{+*})$, the trace operator on 
$\partial\Omega\times \R^{+*}$ and 
$\Omega\times \{0\}$ is more subtle, 
the map 
\begin{equation}\label{trace2}
\left\{
\begin{array}{ll}
H^s(\Omega\times \R_t^{+*})\to \left(\prod_{k< s-1/2} H^{s-1/2-k}(\partial\Omega\times \R_t^{+*})\right)
\times \left(\prod_{k<s-1/2}H^{s-1/2-k}(\Omega \times \{0\})\right),
\\
u\to (\partial_n^ku|_{\partial\Omega
\times \R^{+*}},\ \partial_t^ku|_{\Omega\times \{0\}})_{k<s-1/2},
\end{array}\right.
\end{equation}
is continuous but not surjective: if 
$s\notin \N$, 
\emph{local compatibility 
conditions} between $(g_k,v_k)\in 
(\prod H^{s-1/2-k}(\partial\Omega\times \R_t^{+*}))\times (\prod H^{s-1/2-k}(\Omega\times \{0\}))$ are required 
as follows 
\begin{equation}\label{compaclassique1}
\forall\,k+j<s-1,\ 
\partial_t^jg_k|_{t=0}=\partial_n^kv_j|_{\partial\Omega},
\end{equation}
(see \cite{lionsmagenes2}).\\
In the case $s=1$, and $\Omega=\R^{d-1}
\times \R^{+*}$, surjectivity requires 
the \emph{global compatibility condition}
\begin{equation}\label{compaclassique2}
v_0(x',t)-g_0(x',t)\in H^{1/2}_{00}(\partial\Omega).
\end{equation}
This condition extends to smooth $\Omega$,
see the short comment after \eqref{CCj-}.
\\
Provided such compatibility conditions are added, the trace 
map is a surjection and has a right inverse, this very well 
known fact will be proved later in the article in some basic 
cases where it is needed with more 
precise estimates.

\section{Regularity for the pure boundary value problem}
Consider the boundary value problem
\begin{equation}\label{bvp}
\left\{
\begin{array}{ll}
 Lu=f,\ (x,t)\in \Omega \times \R_t^+\\ 
 Bu|_{\partial\Omega}=g,\\
 u|_{t=0}=0.
\end{array}
\right.
\end{equation}
When $g$, $f$ can be smoothly extended by 
$0$ for $t<0$, the smoothness of $u$ is well known \cite{Tartakoff},\cite{ChaPi}.
The classical proof is done by 
first studying the pure boundary value 
problem posed on $t\in \R$, the case 
$t\in \R^+$ is then deduced by an extension 
by $0$ for $t<0$. We give here a minor 
variation of this argument that directly 
tackles \eqref{bvp}.
\begin{prop}\label{regBVP}
Let $k\in \N$. 
If the extension of $f$ and $g$ by $0$ for $t<0$ belongs to
$H^k$,  then for $\gamma$ large enough 
 the solution of $\eqref{bvp}$ 
 satisfies $e^{-\gamma t}u\in H^k(\Omega
 \times \R^+_t)$. In particular, its 
 belongs to $H^k(\Omega\times [0,T])$ for 
 any $T>0$.
\end{prop}
\begin{proof}
The classical plan is to straighten the boundary through 
local maps, then use a tangential regularization. It is done 
by induction on $k$, it suffices to prove the final step 
where we assume $u\in H^{k-1}(\R^d\times \R_t)$ and 
prove $u\in H^k$.\\
We fix local maps $\varphi_j$ 
as in assumption \ref{assump1}. 
Let $(\psi_j)_{0\leq j\leq J}$ be a partition of unity associated to 
$\Omega\cup (\cup_j\text{Im}(\varphi_j))$. 
% It is sufficient to prove that for any 
% $j$, $\psi_je^{-\gamma t}u
%  \in H^k$, ou par les r\`egles de composition usuelles 
%  $\Phi_j^*(\psi_je^{-\gamma t}u)\in H^k$, o\`u $^*$ est le pullback. En multipliant 
%  l'\'equation par $e^{-\gamma t}$ on se ram\`ene aux Sobolev usuels, \\
%  L'\'equation transport\'ee est 
%  $$
% \psi_j\left( \partial_tu_j+\sum_i\left(\sum_k A_k(\Phi_j(y))(D_y\Phi_j)^{-1}_{ik}(y)\right)
%  \partial_{y_i}\right)u(\Phi_j(y))=\psi_jf(\Phi_j(y)).
%  $$
\noindent
We denote the new variable $y=(y',y_d)$, 
$u_j=(\psi_je^{-\gamma t}u)
\circ\varphi_j$, and $u_0=\psi_0 u$, 
$L_j=\partial_t+\gamma+\sum_i\left(\sum_k 
A_k(D_y\varphi_j)^{-1}_{ik}(y)\right)
 \partial_{y_i}$. 
For $1\leq j\leq J$, $u_j$ satisfies
 \begin{equation}
 \label{eqredresse}
 \left\{
 \begin{array}{ll}
L_ju_j+([\psi_j,L]e^{-\gamma t}u)\circ \varphi_j=e^{-\gamma t}(\psi_jf)\circ \varphi_j:=f_j,
(y',y_d,t)\in \R^{d-1}\times \R^+\times \R_t^+,
\\
B(\varphi_j(y',0))u_j(y',0,t)=e^{-\gamma t}(\psi_j g)(\varphi_j(y',0),t):=g_j.
 \end{array}
 \right.
 \end{equation}
 For simplicity we still denote $B$ for $B\circ 
 \varphi_j(\cdot,0)$.
% Note that the boundary $y_d=0$ is not characteristic for 
% $L_j$, indeed the last line of $(D\varphi_j)^{-1}$
%  is parallel to the normal, say $\lambda \mathrm{n}$, 
%  $\lambda\neq 0$. and the coefficient of 
%  $\partial_{y_d}u_j$ is 
%  $$
%  \sum_{k=1}^q A_k(D_y\Phi)^{-1}_{dk}=\lambda 
%  \sum A_k n_k\text{, invertible}.
%  $$ 
The regularization procedure was introduced by H\"ormander 
\cite{hormanderBVP}: for $v\in L^2(\R^p),\ p\geq 1$, define
$$
\|v\|_{H^{s,\delta}(\R^p)}^2=\int_{\R^d}|\widehat{v}(\xi)|^2
\frac{(1+|\xi|^2)^{s+1}}{1+|\delta \xi|^2}d\xi
\longrightarrow_{\delta \to 0} \|v\|_{H^{s+1}}^2.
$$
Let $\rho(x)\in C_c^\infty(\R^{p})$, 
such that $|\widehat{\rho}(\xi)|\lesssim |\xi|^{m},
\ m>k$ and $\widehat{\rho}$ does not cancel on a neighborhood outside $0$ (such functions are easily constructed, for example 
using $\Delta^{m/2}(\rho_0(t)\rho'(y'))$, $m$ even). 
Define $\rho_\varepsilon=\rho(\cdot/\varepsilon)/\varepsilon^d$.
It is an exercise in calculus that for $0\leq s\leq k-1$, 
an equivalent norm to 
$\|\cdot\|_{H^{s,\delta}}$ -\emph{uniformly in} $\delta$- is 
\begin{equation}\label{equivsobo}
\|v\|_{L^2}+\left(\int_0^1\|v*\rho_\varepsilon\|_{L^2}^2\frac{1}
{\varepsilon^{2(s+1)}(1+\delta^2/\varepsilon^2)}\frac{d\varepsilon}{\varepsilon}\right)^{1/2}\sim 
\|v\|_{H^{s,\delta}}.
\end{equation}
Friedrich's lemma can be generalized in such settings: 
for $P$ a first order differential operator with smooth 
coefficients
\begin{equation}\label{friedrichs}
\int_0^1\|[P,\rho_\varepsilon*]v\|_{L^2}^2\frac{1}
{\varepsilon^{2(s+1)}(1+\delta^2/\varepsilon^2)}\frac{d\varepsilon}{\varepsilon}\lesssim \|v\|_{H^{s,\delta}}^2.
\end{equation}
For details, we refer to \cite{ChaPi} chapter $2$ section $6$.\\
We shall use tangential mollifiers $\rho_\varepsilon(x',t)$
for the functions $u_j$, $1\leq j\leq J$, and full mollifiers 
$\rho_\varepsilon(x,t)$ for $u_0$.
Everything in \eqref{eqredresse} is extended by $0$ for $t<0$. 
Note that due to the assumptions on $f,g$, the extensions of $(f_j,g_j)$
are still in $H^k$. We apply $\rho_\varepsilon*$ to 
\eqref{eqredresse} for $1\leq j\leq J$:
\begin{equation}\label{eqreg}
\left\{
\begin{array}{ll}
 L_j\rho_\varepsilon*u_j=\rho_\varepsilon*f_j-\rho_\varepsilon*[\psi_j,L_j]e^{-\gamma t}u\circ 
 \varphi_j-[\rho_\varepsilon*,L_j]e^{-\gamma t}u_j,\\
 B(\rho_\varepsilon*u_j)|_{y_d=0}=\rho_\varepsilon*g_j
 -[\rho_\varepsilon*,B]u_j|_{y_d=0}.
\end{array}
\right.
 \end{equation}
Since $\rho_\varepsilon*u_j$ belongs to 
$L^2(\R^+,H^\infty(\R^{d-1}\times \R_t)$ 
it is in $H^\infty$ due to non-characteristicity,  
we can use the resolvent 
estimate \eqref{resolv}.
% Clairement pour $\delta>0$, cette quantit\'e est finie ssi $u\in H^{k-1}$, et $\|u\|_{H^k}=\lim_{\delta\to 0}
% \|u\|_{k-1,\delta}$. On peut 
% v\'erifier par calcul que 
% $$
% \|u\|_{k-1,\delta}^2\sim \|u\|_{k-1}^2+\int_0^1|u*\rho_\varepsilon|^2\varepsilon^{-2k}(1+\delta^2/\varepsilon^2)^{-1}\frac{d\varepsilon}{\varepsilon}.
% $$
\begin{eqnarray*}
\gamma \|\rho_\varepsilon*u_j\|_{L^2}^2
+|\rho_\varepsilon*u_j|_{L^2}^2
&\lesssim &
\frac{\|\rho_\varepsilon*f_j\|_{L^2}^2
+\|\rho_\varepsilon*[\psi_j,L_j]e^{-\gamma t}u\circ\varphi_j\|_{L^2}^2+\|[\rho_\varepsilon*,L_j]u_j\|^2_{L^2}}{\gamma}\\
&& +|\rho_\varepsilon*g_j-[\rho_\varepsilon*,B]u_j|_{L^2}^2.
\end{eqnarray*}
Multiplying by $\varepsilon^{-2k-1}
\left(1+(\delta/\varepsilon)^2
\right)^{-1}$, integrating in $\varepsilon$
and using Friedrich's lemma we have
\begin{equation}\label{estimtan}
\gamma \|u_j\|_{L^2H^{k-1,\delta}}^2+ 
|u_j|_{H^{k-1,\delta}}^2\lesssim 
\frac{\|f_j\|_{H^k}^2+\|[\psi_j,L_j]
e^{-\gamma t}u\circ \varphi_j\|_{L^2H^{k-1,\delta}}^2}{\gamma}+ \|g_j\|_{H^k}^2.
\end{equation}
The commutator $[\psi_j,L_j]$ is the 
multiplication by a smooth matrix 
$\theta_j$. Due to the special 
structure of the local maps, 
$\varphi_i^{-1}
\circ\varphi_j$ has the form 
$\left(\varphi_{i,j}(y'),y_d\right)$
hence 
$$
\theta_j e^{-\gamma t}u\circ \varphi_j=
\sum_1^J \psi_i\theta_ju_i(\varphi_{i,j}(y'),y_d)+\theta_ju_0\circ \varphi_j.
$$
Thanks to composition rules (in $H^{s,\delta}$, again see 
\cite{ChaPi}), 
$$
\|[\psi_j,L_j]e^{-\gamma t}u\circ \varphi_j\|_{L^2H^{k-1,\delta}}\lesssim \sum_{i=1}^J \|u_i\|_{L^2H^{k-1,\delta}}
+\|u_0\|_{H^{k-1,\delta}}
$$
For $\gamma$ large enough, this can be 
absorbed in (the sum over $j$ of) the left-hand 
side of \eqref{estimtan}:
\begin{equation}\label{estimtan2}
\sum_{j=1}^J\gamma \|u_j\|_{L^2H^{k-1,\delta}}^2+ 
|u_j|_{H^{k-1,\delta}}^2\lesssim 
\frac{\sum_1^J\|f_j\|_{H^k}^2+\|u_0\|_{H^{k-1,\delta}}^2}{\gamma}+ \sum_1^J|g_j|_{H^k}^2.
\end{equation}
It seems ``moral'' that noncharacteristicity should imply 
the same bound for $\|u_j\|_{H^{k-1,\delta}}$, however 
the $H^{k-1,\delta}$ norm is a non local norm for functions 
defined on $\R^d\times \R_t$, hence such an assertion is not clear. Instead we first obtain interior estimates with 
similar, simpler computations
\begin{equation}\label{interieur}
\gamma \|u_0\|_{H^{k-1,\delta}}^2
\lesssim \frac{\|f_0\|_{H^k}^2
+\|e^{-\gamma t}\widetilde{\psi_0}u\|_{H^{k-1,\delta}}^2}{\gamma},\ \text{supp}(\widetilde{\psi_0)}\subset \Omega,\ 
\widetilde{\psi_0}\equiv 1\text{ on }
\text{supp}(\psi_0).
\end{equation}
Decomposing again $\widetilde{\psi_0}u
=\sum_{j=0}^J \widetilde{\psi_0}\psi_ju$, 
and following the same lines that led  to 
\eqref{estimtan2}, 
\begin{eqnarray}
\sum_{j=1}^J\gamma \|\widetilde{\psi_0} 
\psi_j u\circ \varphi_j\|_{L^2H^{k-1,\delta}}^2
&\lesssim &
\frac{\sum_1^J\|f_j\|_{H^k}^2+\|u_0\|_{H^{k-1,\delta}}^2}{\gamma}+\sum_1^J|g_j|_{H^k}^2.
\end{eqnarray}
A simple consequence of the definition of the $H^{s,\delta}$
spaces is that for any tangential differential operator $D$
of order $1$ and $s\geq 1$
\begin{equation}\label{transfert}
\|Dv\|_{H^{s-2,\delta}}\lesssim 
\frac{1}{C}\|v\|_{H^{s-1,\delta}}+C\|v\|_{L^2H^{s-1,\delta}}.
\end{equation}
Now for $j\geq 1$, each function 
$\widetilde{\psi_0}\psi_ju\circ \varphi_j$ 
is compactly supported in $\R^{d-1}\times \R^{+*}\times \R_t$,
and on its support $L_j$ is (uniformly) non characteristic, 
so we may extend it by zero for $y_d<0$ and use 
\eqref{transfert} to deduce
\begin{equation}\label{interieur2}
\sum_{j=1}^J\gamma \|\widetilde{\psi_0} 
\psi_j u\circ \varphi_j\|_{H^{k-1,\delta}}^2
\lesssim 
\frac{\|f_j\|_{H^k}^2+\|u_0\|_{H^{k-1,\delta}}^2}{\gamma}
+\sum_1^J|g_j|_{H^k}^2
+\gamma \|u\|_{H^{k-1}_\gamma}^2.
\end{equation}
Note that the term 
$\gamma\|e^{-\gamma t}u\|_{H^{k-1}_\gamma}^2$
is present due to the factor 
$\gamma$ in the definition of 
$L_j$. Thanks to the 
 induction assumption, this lower 
 order term is bounded by $\|g\|_{H^{k-1}_\gamma}^2+\|f\|_{H^{k-1}_\gamma}^2$.
Putting together \eqref{estimtan2},
\eqref{interieur}, \eqref{interieur2}
we have 
$$
\left(\sum_1^J\|u_j\|_{L^2H^{k-1,\delta}}^2+\|u_0\|_{H^{k-1,\delta}}^2\right)
+\sum_1^J|u_j|_{H^{k-1,\delta}}^2
\lesssim \|e^{-\gamma t}f\|_{H^k}^2
+|e^{-\gamma t}g|_{H^k}^2.
$$
Letting $\delta\to 0$ we have $u_j\in L^2H^k,\ 1\leq j\leq J$
and $u_0\in H^k$.  We conclude that $u\in H^k$ again thanks to 
the uniform non characteristicity.
\end{proof}
\section{Smoothness of the IBVP: the 
integer case}
We assume in this section that 
$(u_0,g,f)\in (H^k)^3$ satisfy 
the compatibility conditions \eqref{CCj} up to order
$k$, and we prove theorem \ref{mainth} in these
settings.\\
To prove that $u\in \cap_{j=0}^k C_t^jH^{k-j}$, the strategy is to use the regularity 
for the pure boundary value problem by substracting an approximate solution 
(actually a Taylor expansion at $t=0$) 
to $u$. For technical reasons, it is 
necessary to use much more regular data
that satisfy compatibility conditions to 
higher order. The construction of such 
data requires the following lifting lemma 
that is also used in the next section.
\begin{lemma}\label{releve}
For $m\in \N$, there exists a 
lifting map $R_m:\ H^{s}(\partial\Omega)\to H^{m+s+1/2}(\partial\Omega\times \R_t),$
continuous for any $s>0$
such that 
\begin{equation}\label{trace}
\partial_t^{m}R_mg|_{t=0}=g,\ \partial_t^jR_mg|_{t=0}=0,\ j<m+s,
\ j\neq m-1, 
\end{equation}
and for $r<m+1/2$, $|\|R_m|\|_{L^2\to H^r}<<1$ is arbitrarily small.
\end{lemma}
\begin{proof}
Up to the use of local maps, 
the problem is reduced to $\partial\Omega=\R^{d-1}$, 
and to construct a lifting valued in 
$H^{m+s+1/2}(\R^{d-1}\times \R_t)$. The variables are
denoted $(x',t)$.\\
We choose $\chi\in C_c^\infty(\R)$ such 
that $\chi^{(k)}(0)=0,\ k\neq m,\ \chi^{(m)}(0)=1$.
We use the Fourier transform on $\R^{d-1}\times \R_t$ 
and denote $\xi$ the dual
variable of $x'$, $\tau$ the dual variable of $t$, and 
$\lambda$ is a large parameter to fix later:
$$
\widehat{R_mg}=\frac{\widehat{\chi}(\tau/(\lambda\langle \xi\rangle))}{(\lambda \cxi)^{m+1}}\widehat{g}(\xi),\ 
\text{ equivalently } 
\mathcal{F}_{x'}\left(R_m(g)
(\xi,y)\right)=\frac{\chi(\lambda t\langle \xi\rangle)}{\lambda^{m}\langle \xi\rangle^{m}}\widehat{g}(\xi),\ 
\cxi=\sqrt{1+|\xi|^2}.
$$
The trace relations \eqref{trace} are obvious from the second formula.
The $H^{m+s+1/2}$  norm is easily bounded
\begin{eqnarray*}
\|R_mg\|_{H^{m+s+1/2}(\R^d)}^2&=&\int \frac{|\widehat{\chi}(\tau/(\lambda\langle \xi\rangle))|^2|\widehat{g}|^2}
{(\lambda \cxi)^{2(m+1)}}(\cxi^2+\tau^2)^{m+s+1/2}d\xi d\tau
\\
&=&\int \frac{|\widehat{\chi}(\tau)|^2|\widehat{g}|^2}{(\lambda \cxi)^{2(m+1)}}
(\cxi^2(1+\lambda^2\tau^2))^{m+s+1/2} d\xi \lambda \cxi d\tau
\\
&\leq& \int |\widehat{g}|^2\cxi^{2s}
\int |\widehat{\chi}(\tau)|^2
\frac{(1+\lambda^2\tau^2)^{m+s+1/2}}{\lambda^{2m+1}}d\tau\,d\xi.
\\
&\lesssim &\lambda^{2s}\|g\|_{H^{s}}^2.
\end{eqnarray*}
With the same computation
\begin{eqnarray*}
\|\widehat{R_mg}\|_{H^r}^2&\leq&
\int \frac{|\widehat{g}|^2}{(\lambda\cxi)^{2(m-r)+1}}\int |\widehat{\chi}(\tau)|^2(1+\lambda^2
\tau^2)^rd\tau\,d\xi\lesssim  
\frac{\|g\|_{L^2}^2}{\lambda^{2(m-r)+1}}.
\end{eqnarray*}
It is therefore sufficient to choose 
$\lambda$ large enough to ensure the smallness of $\|R_m\|_{L^2\to H^r}$.
\end{proof}

\begin{lemma}[Construction of smooth compatible data]\label{compahaute}
Let $k\geq 0$, $(u_0,g,f)\in (H^k)^3$ 
satisfying the compatibility conditions up 
to order $k$. For any $m>k$, 
there exists $(u_{0,n},g_n,f_n)\in (H^\infty)^3$
satisfying the compatibility conditions up 
to order $m$, and such that 
$$
\|(u_0,g,f)-(u_{0,n},g_n,f_n)\|_{(H^k)^3}
\to 0.
$$
\end{lemma}
\begin{proof}
By density of smooth functions, there exists a sequence $(u_{0,n},g_n,f_n)\in (H^\infty)^3$ converging to $(u_0,g,f)$
in $(H^k)^3$. We denote $v_{j,n}$ the corresponding 
functions in \eqref{taylor}.
For  $j\geq 1$ 
the ``compatibility error'' is defined as
$$\varepsilon_{j,n}:=
\partial_t^{j-1}g_n|_{t=0}
-\sum_{l=0}^{j-1}\binom{j}{l} 
(\partial_t^lB)v_{j-1-l,n}|_{\partial\Omega}.$$
Due to the compatiblity conditions and continuity 
of traces we have 
$$
\forall\,1\leq j\leq k,\ 
\left\|\varepsilon_{j,n}\right\|_{H^{k-j+1/2}}\longrightarrow_{n} 0.
$$
As a consequence, given a lifting operator $R_{j-1}$
as in lemma \ref{releve}, 
$\left\|R_{j-1}\varepsilon_{j,n}\right\|_{H^k}\to_n 0$.\\
For $k<j\leq m$, 
$\varepsilon_{j,n}$ is not small in any Sobolev space,
nevertheless from lemma \ref{releve} there exists a lifting 
$R_{j-1,n}$ such that $\|R_{j-1,n}\varepsilon_{j,n}\|_{H^k}
\leq 1/n$. We then define
$$\widetilde{g_n}:=g_n-\sum_{j=1}^m
R_{j-1}(\varepsilon_{j,n}).
$$ 
This choice ensures that compatibility conditions are satisfied 
by $(u_{0,n},\widetilde{g_n},f_n)$ up to order $m$ and 
$\|\widetilde{g_n}-g\|_{H^k}
\to 0$.
\end{proof}

\paragraph{Proof of theorem \ref{mainth}
(integer case)} We follow the notations of lemma 
\ref{compahaute}; $v_{j,n}$ are smooth functions 
defined by \eqref{taylor} for smooth 
data $(u_{0,n},g_n,f_n)$. 
We define the approximate solution 
$$
u_{app,n}(x,t)=\sum_{j=0}^{m-1}\frac{t^j}{j!}
v_{j,n}(x)\chi(t),\ \chi\in C_c^\infty
(\R^+),\ \chi\equiv 1\text{ near 0}.
$$
We solve then 
$$
\left\{
\begin{array}{ll}
 Lw_n=f_n-Lu_{app,n},\\
 w_n|_{t=0}=0,\\
 Bw_n=g_n-Bu_{app,n},
\end{array}
\right.
$$
By construction, the data 
$(0,g_n-Bu_{app,n},f_n-Lu_{app,n})$ are 
smooth and it is easily seen that 
$\partial_t^j(g_n-Bu_{app,n})=0,$
$\partial_t^j(f_n-Lu_{app,n})=0,\ j\leq k+1$ provided
$m\geq k+4$. Hence according 
to proposition \ref{regBVP}, the solution
$w_n$ belongs to $H^{k+2}$, this implies by Sobolev embedding 
$w_n\in\cap_{j=0}^{k+1}C_t^jH^{k+1-j}$.
Therefore $u_n:=w_n+u_{app,n}$ is also in $\cap_{j=0}^{k+1}C_t^jH^{k+1-j}$, and it is a solution of \eqref{IBVP} with data 
$(u_{0,n},g_n,f_n)$.\\
Using a differentiation argument similar to the proof 
of proposition 
\ref{regBVP}, but  much simpler since no 
regularization is needed,
we see that $u_n$ satisfies \ref{semigreg}:
\begin{eqnarray*}
\sum_{j=0}^k \|\partial_t^j(e^{-\gamma t}u_n)\|_{C(\R^+,H^{k-j}(\Omega))}+
 |e^{-\gamma t}u_n|_{\partial\Omega}|_{H^k}&\lesssim &\bigg(\|u_{0,n}\|_{H^k(\Omega)}+
 |e^{-\gamma t}g_n|_{H^k(\partial\Omega\times [0,T]}
 \\
 && \hspace{35mm}+\|e^{-\gamma t}f_n\|_{H^k}\bigg),
\end{eqnarray*}
as well as \eqref{resolvreg}.
The same estimates, applied to $u_p-u_q,\ (p,q)\in \N^2$,
shows that $(u_n)$ is 
a Cauchy sequence in $\cap_{j=0}^kC_t^jH^{k-j}$, but since 
$(u_n)$ converges (in $L^2$) to the solution $u$ of \eqref{IBVP} with data 
$(u_0,g,f)$, this ensures that $u\in 
\cap_{j=0}^kC_t^jH^{k-j}$. The estimate \eqref{resolvreg}
is then an elementary 
differentiation argument : tangential 
regularity is obtained directly by differentiation 
(which is now legal) and use of the $L^2$ estimate, 
while normal regularity uses the non characteristicity.
% 
% 
% Existence d'une solution approch\'ee : partons de 
% $\partial_tu-Au=f$, par calcul direct on a 
% $\partial_t^ju=B_ju+C_jf$, 
% $B_j$ op\'erateur diff\'erentiel spatial d'ordre $j$, 
% $C_j$ d'ordre $j-1$. On pose 
% $$
% u_{app}=\sum_{j=0}^N(B_ju_{0,n}+C_jf_n|_{t=0})
% \frac{t^j}{j!}\chi(t),\ \chi\in C_c^\infty,
% \chi\equiv 1\text{ pr\`es de }0.
% $$
% On a $u_{app}\in H^\infty(\Omega\times \R_t)$, 
% $$
% \left\{
% \begin{array}{ll}
%  L(u_n-u_{app})=f_n-Lu_{app},\\
%  u_n-u_{app}|_{t=0}=0,\\
%  B(u_n-u_{app}=g_n-Bu_{app},
% \end{array}
% \right.
% $$
% Par calcul, on voit que 
% $\partial_t^j(g_n-Bu_{app})=0,$
% $\partial_t^j(f_n-Lu_{app})=0,\ j\leq N$
% donc 
% $(g_n-Bu_{app},f_n-Lu_{app})\in H^{N+1}_0$. 
% Par r\'egularit\'e pour le probl\`eme
% aux limites pur (proposition \ref{regBVP}) 
% on a donc $u_n-u_{app}
% \in H^{N+1}(\R^+\times \Omega)$ $(u_n-u_{app}$ est solution r\'eguli\`ere par 
% unicit\'e de la solution $L^2$).\\
% En cons\'equence, $u_n\in H^{N+1}$, et l'estim\'ee a priori $L^2$ est disponible.
% Par les manipulations habituelles de d\'erivation tangentielle et noncaract\'eristicit\'e du bord, on obtient l'estim\'ee de semi groupe $H^N$
% $$
% \sum_{j=0}^N \|\partial_t^ju_n\|_{L^\infty_TH^{N-j}}
% +\|u_n\|_{H^N(\partial\Omega\times[0,T]}
% \lesssim \|u_{0,n}\|_{H^N}+\|g_n\|_{H^N}+\|f_n\|_{H^N}.
% $$
% Pour $N=k$, l'estimation appliqu\'ee \`a 
% $u_p-u_q$ montre que $u_n$ est une 
% suite de Cauchy $L^\infty_TH^k$, et donc 
% $u\in L^\infty_TH^k$ (unicit\'e de la limite). La preuve est d\'ej\`a termin\'ee.
% Apparemment $N=k$ suffisait.
\section{Regularity for positive s}
For ease of presentation, we only detail the 
case $\Omega =\R^{d-1}\times \R^+$. The general case 
can be obtained by using a partition of unity as in the 
previous section. \\
In this section, we follow the (non standard) convention that 
$H^s_0$ is $H^{1/2}_{00}$ if $s=1/2$.\\
Under such settings, we can assume that 
$A_d$ is invertible and $A_d^{-1}$ is
uniformly bounded.
Furthermore since $B:\ \R^p\to 
\R^b$ has maximal rang $b$, there exists 
a smooth basis of 
$\text{Ker}\,B$ (as a smooth vector bundle over the contractible space $\R^{d-1}\times \R_t^+$) 
that we denote $(k_1,\cdots k_{p-b})$. A basis 
$(v_j)_{1\leq j\leq b}$ of
$(Ker B)^\perp$ is then obtained easily: 
$$\widetilde{B}=
\begin{pmatrix}
B\\
k_1^t\\
\vdots\\
k_{p-b}^t
\end{pmatrix}\text{ is an isomorphism }
\R^p\to \R^p,\text{ we can choose }
v_j=\widetilde{B}^{-1}(e_j),\ 1\leq j\leq b.$$
We remind that compatibility conditions of 
order $s=k+\theta,\ k\in \N^*,\ 0<\theta<1$
are defined as follows:
\begin{enumerate}
 \item If $\theta <1/2$, then compatibility 
 conditions \eqref{CCj} up to order $k$ are 
 satisfied.
 \item If $\theta>1/2$, then compatibility 
 conditions \eqref{CCj} up to order $k+1$ 
 are satisfied.
 \item If $\theta=1/2$, compatibility 
 conditions up to order $k$ are satisfied and
 $$
 \int_{\R^{d-1}} \left|\partial_t^{k-1} g(x',y)-\sum_{j=0}^{k-1} \binom{k-1}{j}(\partial_t^jB)\left(A_{k-1-j}u_0+B_{k-1-j}f|_{t=0}\right)|(x',y)\right|^2
 \frac{dy}{y}<\infty.
 $$
\end{enumerate}
\paragraph{The case $0<s<1$}
From the previous section, the map 
$(u_0,g,f)\to u$ solution of \eqref{IBVP} 
is continuous 
\begin{eqnarray*}
X_0\times L^2&:=&(L^2)^3\to C_tL^2\text{ and }\\
X_1\times H^1&:=&\{(u_0,g)\in (H^1)^2:\ Bu_0|_{\partial\Omega}=g|_{t=0}\}\times H^1
\to C_tH^1\cap C^1_tL^2.
\end{eqnarray*}
Let us define for $0\leq\theta \leq 1$
$$
X_\theta=\left\{(u_0,g)\in (H^\theta)^2:\ \text{the compatibility condition of order }\theta\text{ is satisfied}\right\},
$$
(note that compatibility conditions of order less than $3/2$
do not involve $f$).\\
Both the semi-group estimate \eqref{semigreg} and the resolvent 
estimate \eqref{resolvreg} follow from an interpolation 
argument if we can prove that 
\begin{equation}\label{interpX}
X_\theta=[X_0,X_1]_\theta.
\end{equation}
More precisely, since the resolvent estimate implies
for $s=0,1$
\begin{eqnarray*}
\gamma\|u\|_{L^2_\gamma}^2+\|u|_{\partial\Omega}\|^2_{L^2_\gamma}\lesssim \|(u_0,e^{-\gamma t}g)\|_{X_0}^2
+\frac{\|f\|_{L^2_\gamma}^2}{\gamma}\\
\gamma\|u\|_{H^1_\gamma}^2+\|u|_{\partial\Omega}\|^2_{H^1_\gamma}\lesssim C(\gamma)\|(u_0,e^{-\gamma t}g)\|_{X_1}^2
+\frac{\|f\|_{H^1_\gamma}^2}{\gamma},
\end{eqnarray*}
the interpolation identity \eqref{interpX} implies
\begin{equation}\label{resolvsupersharp}
\gamma\|u\|_{H^\theta_\gamma}^2+\|u|_{\partial\Omega}\|^2_{H^\theta_\gamma}\lesssim C'(\gamma)\|(u_0,e^{-\gamma t}g)\|_{X_\theta}^2
+\frac{\|f\|_{H^\theta_\gamma}^2}{\gamma}.
\end{equation}
(a better estimate would require to 
use weighted $X^\theta$ spaces, a course 
that we chose not to follow).
\paragraph{Proof of \eqref{interpX}}         
We extend $\widetilde{B}$ on $\Omega\times \R_t^+$ as $\widetilde{B}(x',y,t)=\widetilde{B}(x',t)$, and 
consider the map $u_0\to \widetilde{B}u_0:=\widetilde{u_0}$. It is an isomorphism 
$(H^s(\Omega))^p\to (H^s(\Omega))^p$, and the compatibility condition 
can be rewritten
$$
Bu_0|_{\partial\Omega}=g|_{t=0}
\Leftrightarrow B\widetilde{B}^{-1}
\widetilde{B}u_0|_{\partial\Omega}
=g|_{t=0}\Leftrightarrow 
\begin{pmatrix}
 I_b & 0
\end{pmatrix}
\widetilde{u_0}|_{\partial\Omega}
=g|_{t=0},\text{ 
with }\widetilde{u_0}=\widetilde{B}u_0.
$$
This transformation ``diagonalizes'' \eqref{interpX}, and we are 
reduced to determine
$$
[L^2\times L^2,\ H^1\times H^1]_\theta,
\text{ and }
\left[L^2\times L^2,\ \{(u_0,g)\in H^1\times H^1:\ 
u_0|_{y=0}=g|_{t=0}\}\right]_\theta=[Y_0,Y_1]_\theta,
$$ 
where $u_0$ and $g$ are now \emph{scalar}
functions.\\
Of course, it is well-known that 
$[L^2,H^1]_\theta=H^\theta$, so the first 
case is immediate.  In the second case,
surprisingly, we were not able to find 
results in the litterature except in the 
simplest case $\theta<1/2$, which is in
\cite{lionsmagenes2} section 14.
\begin{lemma}\label{interpfacile}
For $\theta<1/2$, $[Y_0,Y_1]_\theta=Y_\theta$.
\end{lemma}
\begin{proof}
The following inclusions are clear :
$H^1_0\times H^1_0\subset Y_1\subset H^1(\Omega)
\times H^1(\partial\Omega\times \R^+)$. 
On the other hand, for $\theta<1/2$ we have 
$[L^2,H^1_0]_\theta=H^\theta$ (\cite{lionsmagenes}, chapter $1$ section $11$), and we can conclude 
$$
H^\theta\times H^\theta=[L^2\times L^2,H^1_0\times H^1_0]_\theta\subset 
[Y_0,Y_1]_\theta\subset 
[L^2\times L^2,H^1\times H^1]_\theta=
H^\theta\times H^\theta.
$$
\end{proof}
\begin{lemma}\label{relevecoin}
For $0<\theta\leq 1$, there exists an universal (independent of $\theta$) 
operator $R$
 $$
 R:\ Y_\theta\to H^{\theta+1/2}(\Omega\times \R^+),\ \forall\, 
 0<\theta\leq 1.
 $$
\end{lemma}
\begin{proof}
This is a result due to Grisvard \cite{Grisvard}, for completeness we include 
a simple proof. Given $(u_0,g)\in (H^\theta)^2$, from lemma \ref{releve} there 
exists an opertor $R_b:\ g\to R_b(g)\in H^{\theta+1/2}$ which is independent of 
$\theta$. By construction, $R_bg|_{t=0}-u_0\in H^\theta_0$. If $\theta=1/2$, we also notice 
$$
R_bg(x',y,0)-u_0(x',y)=
\underbrace{R_bg(x',y,0)-g(x',y)}_{H^{1/2}_{00}\text{ by interpolation}}+\underbrace{g(x',y)-u_0(x',y)}_{H^{1/2}_{00}
\text{by assumption}}.
$$
If there exists an universal lifting 
$R_0:\ H^\theta_0(\Omega)\to \{u\in 
H^{\theta+1/2}(\Omega\times \R^+)|\ 
 u|_{\partial\Omega}=0\}$, $R$ can be
 defined as 
 $R(u_0,g)=R_bg+R_0(u_0-R_bg|_{t=0})$ so we 
 focus on the construction of $R_0$.\\ 
 For $u_0\in H^\theta_0$ $(H^{1/2}_{00}$ for 
 $\theta=1/2$), 
we extend it as an odd function of $y$,
$I(u_0)$ defined on $\R^d$. The map $I:\ H^\theta_0(\R^{d-1}\times \R^+)\to H^\theta(\R^d)$ is continuous as
it is clearly the case for $\theta=0,1$. 
Define now 
 $$
\widehat{R_I(I(u_0))}(\xi,\delta)=\chi(\cxi t)\widehat{I(u_0)}(\xi),
 $$
where $\chi$ is as in lemma \ref{releve}. According to the proof 
of lemma \ref{releve},
$R_I\circ I:\ H^\theta\to 
 H^{\theta+1/2}(\R^d\times \R^+)$ is 
continuous, moreover by construction 
 $R_I\circ I(u_0)$ is an odd function of $y$, therefore 
 necessarily
 $R_I\circ I (u_0)|_{y=0}=0$. Thus by taking 
 the restriction on $\R^{d-1}\times \R^+_y
 \times \R_t^+$, 
 $R_0:=R_I\circ I$ solves the problem.
\end{proof}
\begin{prop}
For $0<\theta<1$, $[Y_0,Y_1]_\theta=Y_\theta$.
\end{prop}
\begin{proof}
On one hand, the map 
$(u_0,g)\to u_0(x',y)-g(x',y)$ 
is continuous $Y_i\to H^i_0$ for $i=0,1$, 
therefore by interpolation it is continuous 
$[Y_0,Y_1]_\theta\to H^\theta_0$.
This gives the first inclusion 
\begin{equation}
[Y_0,Y_1]_\theta\subset Y_\theta.\label{inclusion} 
\end{equation}
On the other hand, from Lions-Peetre 
reiteration theorem, for any $0<s,\theta<1$
$$
[[Y_0,Y_1]_s,Y_1]_\theta=[Y_0,Y_1]_{\theta+s(1-\theta)}.
$$
If we have for some 
$s<1/2$, $[Y_s,Y_1]_\theta\supset Y_{\theta+s(1-\theta)}$ for any $0<\theta<1$, then by reiteration 
this implies $[Y_0,Y_1]_\theta=Y_\theta$ 
for $\theta >s$. On the other hand, 
the case 
$\theta\leq s$ is contained in lemma 
\ref{interpfacile}.\\
For any $0<r<1$ we define the map 
$$
u\in H^{r+1/2}(\Omega\times \R_t^+)
\to \text{Tr}(u)=(u|_{t=0},u|_{y=0}).
$$
It is easily seen that $\text{Tr}$ is continuous 
$H^{3/2}\to Y_1$ and $H^{1/2+s}\to Y_s$ for $0<s<1/2$.
As it is well known that
$[H^{s+1/2},H^{3/2}]_{\theta}=H^{1/2+\theta +(1-\theta)s}$,
we deduce by interpolation 
$$
\text{Tr}:\ H^{1/2+(1-\theta)s+\theta}=[H^{s+1/2},H^{3/2}]_\theta\to [Y_s,Y_1]_\theta\text{ is continuous}.
$$
We observe now that the lifting $R$ from lemma \ref{relevecoin}
is a right inverse for $\text{Tr}$: for fixed
$0<s<1/2$ and any $0<\theta<1$, we have $\text{Tr}\circ R=I_d:\ Y_{\theta+s(1-\theta)}
\to Y_{\theta+s(1-\theta)}$. Since $R$ maps 
$Y_{\theta+s(1-\theta)}$ to $H^{s(1-\theta)+\theta+1/2}$, this 
implies
$$Y_{\theta+s(1-\theta)}\subset [Y_s,Y_1]_\theta,$$ 
which was the required converse inclusion.
\end{proof}
% \begin{rmq}
% We shall use in the sequel the simplified 
% resolvent estimate :
%  $$
%  \|e^{-\gamma t}u\|_{H^i(\Omega_T)}\lesssim \frac{1}{\gamma}
%  \big(\|e^{-\gamma t}f\|_{H^i}+\|e^{-\gamma t}g\|_{H^i}
%  +\|u_0\|_{H^i}\big),\ i=0,1.
%  $$
% By interpolation this implies 
%  \begin{equation}\label{resolvS}
%  \|u\|_{H^s(\Omega_T)}\lesssim \frac{1}{\gamma}
%  \big(\|(L+\gamma)u\|_{H^s}+\|g\|_{H^s}
%  +\|u_0\|_{H^s}\big),\ 0\leq s\leq 1.
% \end{equation}
% 
% \end{rmq}

\paragraph{The case $s>1$} We denote $s=k+\theta$, 
$0\leq \theta<1$. According to the integer
case, we already have $u\in \cap C^{k-j}H^j$.
For any tangential multi-index $\alpha$ of 
order $k$ (that is, $\alpha_d=0,|\alpha|=k$),
$\partial^\alpha u$ satisfies
\begin{equation}\label{derBVP}
\left\{
\begin{array}{ll}
L(\partial^\alpha u)=\partial^\alpha f+[L,\partial^\alpha]u,\\
B\partial^\alpha u|_{\partial\Omega}=\partial^\alpha g+[B,\partial^\alpha]u|_{\partial\Omega},\\
\partial^\alpha u|_{t=0}=L_\alpha (u_0)+
L_\alpha' (f)|_{t=0}.
\end{array}
\right.
\end{equation}
where $L_\alpha,L_\alpha'$ are 
differential operators of respective order
$\alpha,\alpha-1$. Regularity will again 
be obtained by regularization of the data, we
distinguish three cases:
\subparagraph{The case $0<\theta<1/2$}
With the same argument as in the integer 
case (note that the condition $\theta<1/2$ 
allows to use lemma \ref{releve}), there exists regularized data 
$(u_{0,n},g_n,f_n)\in (H^{k+1})^3$, 
converging to $(u_0,g,f)$ that satisfy 
the compatibility conditions up to order 
$k+1$. The corresponding solution $u_n$ 
belongs to $\cap_0^{k+1} C_t^jH^{k+1-j}$ 
so that we may apply the resolvent estimate 
\eqref{resolvreg} to $\partial^\alpha u_n$ 
with $s=\theta$, combined with
basic trace estimates and the commutator estimate
$\|[\partial^\alpha,L]u_n\|_{H^\theta_\gamma}\lesssim 
\|u\|_{H^s_\gamma}$:
\begin{eqnarray}\nonumber
\gamma\|\partial^\alpha u_n\|_{H^\theta_\gamma}^2
&\lesssim& \|u_{0,n}\|_{H^s}^2+\|g_n\|_{H^s_\gamma}^2
+\frac{\|f_n\|_{H^s_\gamma}^2+
\|u_n\|_{H^s_\gamma}^2}{\gamma}.
\label{keyresolv}
\end{eqnarray}
Due to the boundary being non 
characteristic, we deduce as for the integer
case (note that the fractional regularity 
gained here includes conormal regularity)
for $\gamma$ large enough only depending on
$s$
$$
\gamma \|u_n\|_{H^s_\gamma}^2
\lesssim \|u_{0,n}\|^2_{H^s}+\|g_n\|_{H^s_\gamma}^2+
\frac{\|f_n\|_{H^s_\gamma }^2}{\gamma}.
$$
With the resolvent estimate available, the semi group estimate is 
now an immediate consequence of the case $0<s<1$ applied to 
\eqref{derBVP}:
\begin{eqnarray*}
\|e^{-\gamma t} \partial^\alpha u_n\|_{C_tH^\theta}^2&\lesssim& 
\|u_{0,n}\|_{H^s(\Omega)}^2
+\|f_n\|_{H^s_\gamma([0,T]\times \Omega)}^2
+\|[\partial^\alpha,L]u_n\|_{H^\theta_\gamma}^2
+\|g_n\|_{H^s_\gamma}^2
\\
&\lesssim& 
\|u_{0,n}\|_{H^s(\Omega)}^2
+\|f_n\|_{H^s_\gamma([0,T]\times \Omega)}
+\|g_n\|_{H^s_\gamma}^2.
\end{eqnarray*}
Once more, normal regularity is then obtained thanks to the 
boundary being non characteristic. \\
Letting $n\to \infty$, we deduce that $e^{-\gamma t}u$ is 
in $H^s(\R^+\times \Omega)\cap (\cap_{j=0}^kC^j(\R^+,H^{s-j}(\Omega))$ and satisfies the semi group estimate and the resolvent 
estimate.
\subparagraph{The case $1/2<\theta<1$} This can be done with 
exactly the same argument. Actually, the construction of 
regularized data $(u_{0,n},g_n,f_n)\in (H^{k+1})^3$ 
that satisfy compatibility conditions up to order $k+1$ 
and converging to $(u_0,g,f)$ in $(H^s)^3$ is even simpler. Indeed 
$(u_0,g,f)$ satisfy compatibility conditions up to order $k+1$, 
hence  any regularization of $(u_0,g,f)$ satisfies 
$$
\forall\,1\leq j\leq k+1,\ 
\left\|\underbrace{\partial_t^{j-1}g_n|_{t=0}
-\sum_{l=0}^{k-1}\binom{j}{l} 
(\partial_t^lB)v_{j-1-l,n}|_{\partial\Omega}}_{:=\varepsilon_{j,n}}\right\|_{H^{s-j+1/2}}\longrightarrow_{n} 0,
$$
and it suffices to modify $g_n$ as $g_n-\delta_n$ where $\delta_n$
is a function in $H^{k+1}(\partial\Omega\times\R_t^+)$ that satisfies
for $1\leq j\leq k+1$,
$\partial_t^{j-1}\delta_n|_{t=0}=\varepsilon_{j,n}$
\subparagraph{The case $\theta=1/2$}
\label{cas1/2}
When $s=k+1/2$, the compatibility conditions 
are satisfied in particular up to 
order $k$. From the previous study, 
we have $e^{-\gamma t}u\in (\cap_{j=0}^k
C_t^jH^{k+\theta-j})\cap 
H^{k+\theta}$ for any 
$\theta<1/2$, with the estimate 
$$
\|e^{-\gamma t}u\|_{(\cap_{j=0}^k
C_t^jH^{j+\theta-j})}
+\|e^{-\gamma t}u\|_{H^{k+\theta}}
\leq C(\theta)\|(u_0,g,f)\|_{(H^s)^3}.
$$
Of course this is not enough to conclude,
but the estimate can be sharpened: apply
estimate \eqref{resolvsupersharp} to \eqref{derBVP}
for $\theta<1$ and any tangential multi-index 
$\alpha\in \N^d$, $|\alpha|=k$, 
this reads
\begin{eqnarray*}
\gamma\|\partial^\alpha u\|_{H^\theta_\gamma}^2
&\lesssim& \left\|\left(L_\alpha u_{0}+L_\alpha'f|_{t=0},e^{-\gamma t}(\partial^\alpha g+[B,\partial^\alpha]u|_{\partial\Omega})\right)\right\|_{X_\theta}^2
+\frac{\|f\|_{H^{k+\theta}_\gamma}^2+
\|u\|^2_{H^{k+\theta}_\gamma}}{\gamma}.
\end{eqnarray*}
Recall that the compatibility conditions at order $j$ are
are
$$
\forall\,1\leq j\leq k,\ 
\partial_t^{j-1}g|_{t=0}-\sum_{l=0}^{j-1}
\binom{j}{l}(\partial_t^lB)v_{j-1-l}|_{\partial\Omega}=0,
$$
and at order $k+1/2$
\begin{equation*}
\partial_t^{k}g(x',t)-\left(\sum_{l=0}^{k}
\binom{k}{l}(\partial_t^lB)v_{k-1-l}(x',t)\right)
\in H^{1/2}_{00}\left(\R^{d-1}\times (\R^+)\right). 
\end{equation*}
% It is more convenient to rewrite them 
% in term of the weighted function $e^{-\gamma t}u$: define as in \eqref{taylor}
% $\mathcal{A}_\gamma=\mathcal{A}-\gamma$,
% and inductively
% $$
% v_{0,\gamma}=u_0,\ 
% v_{j,\gamma}=\sum_{l=0}^j\binom{j}{l}(\partial_t^l\mathcal{A}_\gamma|_{t=0})
% v_{j-l}+\partial_t^j(e^{-\gamma t}f)|_{t=0}.
% $$
% The compatibility conditions rewrite 
% \begin{equation}\label{CCgamma}
% \forall\,1\leq j\leq k,\ 
% \partial_t^{j-1}(e^{-\gamma t}g)|_{t=0}-\sum_{l=0}^{j-1}
% \binom{j}{l}(\partial_t^lB)v_{j-1-l,\gamma}|_{\partial\Omega}=0,
% \end{equation}
% and 
% \begin{equation}\label{CGgamma}
% \partial_t^{k}(e^{-\gamma t}g(x',t))-\left(\sum_{l=0}^{k}
% \binom{k}{l}(\partial_t^lB)v_{k-1-l,\gamma}(x',t)\right)
% \in H^{1/2}_{00}\left(\R^{d-1}\times (\R^+)\right). 
% \end{equation}
As a consequence, for any $j\leq k+1$ and any $\beta\in \N^{d-1},\ |\beta|=k+1-j$,
\begin{equation}
\partial_{x'}^\beta\partial_t^{j-1}g(x',t)
-\partial_{x'}^\beta\sum_{l=0}^{j-1}
\binom{j-1}{l}(\partial_t^lB)v_{j-1-l}(x',t)\in H^{1/2}_{00}(\R^{d-1}\times \R^+).
\label{compafrac}
\end{equation}
Furthermore, $e^{-\gamma t}u\in H^k(\Omega
\times \R_t^+)$, hence for any 
 multi-index of order $k-1$
\begin{equation}\label{compagratuite}
\|e^{-\gamma t}\partial^\alpha u|_{y=0}
-e^{-\gamma t}\partial^\alpha u|_{t=0}\|_{H^{1/2}_{00}
(\R^{d-1}\times \R^+)}
\lesssim \|e^{-\gamma t}u\|_{H^{k}(\R^{d-1}\times (\R^+)^2)}.
\end{equation}
Now to make \eqref{derBVP} more explicit, 
let us write $\partial^\alpha=
\partial_t^{j}\partial_{x'}^\beta$, $\beta\in 
\N^{d-1},\ |\beta|=k-j$. Then 
$\partial^\alpha u|_{t=0}
=\partial_{x'}^\beta v_{j}\in H^{1/2}(\R^{d-1}\times \R^+)$, the compatibility condition of order $1/2$ 
for \eqref{derBVP} is thus 
$$
e^{-\gamma t}\left(\partial^\alpha g+
[B,\partial^\alpha]u|_{\partial\Omega}\right)
-B\partial_{x'}^\beta v_{j}\in 
H^{1/2}_{00}(\R^{d-1}\times \R^+).
$$
With basic computations, we now check that 
it is implied by \eqref{compafrac},\eqref{compagratuite}:
\begin{eqnarray*}
e^{-\gamma t}(\partial^\alpha g &+&
[B,\partial^\alpha]u|_{\partial\Omega})
-B\partial_{x'}^\beta v_{j}
\\
&=&
e^{-\gamma t}\left(\partial^\alpha g-\partial_{x'}^\beta 
\sum_{l=0}^j\binom{j}{l}(\partial_t^lB)\partial_t^{j-l}u|_{\partial\Omega}\right)
+e^{-\gamma t}B\partial^\alpha u|_{\partial\Omega}-B\partial_{x'}^\beta 
v_{j}
\\
&=&e^{-\gamma t}\partial^\alpha g
-\partial_{x'}^\beta \sum_1^j\binom{j}{l}(\partial_t^lB)
v_{j-l}
-B\partial_{x'}^\beta
v_{j}\\
&&-\partial_{x'}^\beta\sum_{l=1}^j
\binom{j}{l}(\partial_t^lB)
\left(e^{-\gamma t}\partial_t^{j-l}u|_{\partial\Omega}
-v_{j-l}\right)\\
&& -\partial_{x'}^\beta \left(Be^{-\gamma t}\partial_t^j
u|_{\partial\Omega}\right)
+e^{-\gamma t}B\partial^\alpha u|_{\partial\Omega}\\
&=&e^{-\gamma t}\left(\partial^\alpha g
-\partial_{x'}^\beta \sum_{l=0}^j\binom{j}{l}(\partial_t^lB)v_{j-l}\right)
\\
&& -\partial_{x'}^\beta\sum_{l=1}^j
\binom{j}{l}(\partial_t^lB)
\left(e^{-\gamma t}\partial_t^{j-l}u|_{\partial\Omega}-v_{j-l}\right)\\
&&+[B,\partial_{x'}^\beta]e^{-\gamma t}\partial_t^ju|_{\partial\Omega}
-[B,\partial_{x'}^\beta]v_{j}
.
\end{eqnarray*}
For $j\leq k$, due to the compatibility condition 
\eqref{compafrac}, the first line in the 
last equality is in $H^{1/2}_{00}$. The $H^{1/2}_{00}$ 
norm of the second line  
is easily controlled by writing
$$
e^{-\gamma t}\partial_t^{j-l}u|_{\partial\Omega}-v_{j-l}
=e^{-\gamma t}(\partial_t^{j-l}u|_{\partial\Omega}-
v_{j-l})+(1-e^{-\gamma t})v_{j-l},
$$
the first term can be bounded thanks to \eqref{compagratuite}
while for the second we simply use $(1-e^{-\gamma t})/t
\lesssim 1$.
The same argument is used for the third line.
We deduce that for $\theta <1/2$, $\alpha$ 
tangential, $|\alpha|\leq k$
\begin{eqnarray*}
\gamma\|\partial^\alpha u\|^2
_{H^{\theta}_\gamma(\Omega\times \R_t^+)}
&\lesssim &C(\gamma)\left(\|(u_0,g,f)\|_{(H^{k+1/2})^3}
+\left\|g
-\sum_0^k\binom{k}{l}(\partial_t^lB)
v_{k-1-l}\right\|_{H^{1/2}_{00}
(\R^{d-1}\times \R^+)}\right)\\
&&+\frac{\|u\|_{H^{k+\theta}_\gamma}^2}{\gamma}.
\end{eqnarray*}
Using non characteristicity, we recover
\begin{eqnarray*}
\gamma\|u\|^2
_{H^{k+\theta}_\gamma(\Omega\times \R_t^+)}
&\lesssim &\|(u_0,g,f)\|_{(H^{k+1/2})^3}
+\left\|g
-\sum_0^k\binom{k}{l}(\partial_t^lB)
v_{k-1-l}\right\|_{H^{1/2}_{00}
(\R^{d-1}\times \R^+)}.
\end{eqnarray*}
This estimate is uniform in $\theta 
<1/2$, we deduce that the same estimate 
holds for $\theta =1/2$. Finally we deduce that the semi group estimate is true with the same argument as
for the end of the case $0<\theta <1/2$.
This ends the proof of theorem \ref{mainth}.

\bibliographystyle{plain}
\bibliography{biblio}
\end{document}